\documentclass{article}
\usepackage[utf8]{inputenc}
\usepackage[english]{babel}
\usepackage{blindtext}
\usepackage{amssymb}
\usepackage{amsthm}
\usepackage{amsmath}
\usepackage{tikz-cd}
\usepackage{setspace}
\usepackage{hyperref}
\usepackage{cleveref}
\usepackage[margin=1.2in]{geometry}

\usepackage{graphicx}
\graphicspath{ {images/} }

\newcommand{\aut}[1]{\mathrm{Aut} #1}

\theoremstyle{definition}
\newtheorem{definition}{Definition}[section]

\theoremstyle{plain}
\newtheorem{prop}{Proposition}[section]

\theoremstyle{plain}
\newtheorem{theorem}{Theorem}[section]

\theoremstyle{plain}
\newtheorem{corollary}{Corollary}[theorem]

\theoremstyle{plain}
\newtheorem{lemma}[theorem]{Lemma}

\theoremstyle{remark}
\newtheorem{remark}[theorem]{Remark}

\theoremstyle{remark}

\theoremstyle{plain}

\title{On the homotopy type of the space of fiberings of $S^1 \times S^2$ by simple closed curves}
\author{Yi Wang, Jingye Yang}
\date{}

\begin{document}

\maketitle

\begin{abstract}
For most aspherical Seifert-fibered 3-manifolds $M$, the space of Seifert fiberings $SF(M)$ is known to have contractible components \cite{mccullough}. It is also known that the space of Hopf fiberings of the three-sphere is noncontractible \cite{dglmwy}. We provide the second example of a non-aspherical 3-manifold $M$ such that $SF(M)$ has noncontractible components. In particular, we show that certain components of $SF(S^1 \times S^2)$ are homotopy equivalent to a subspace homeomorphic to the identity-based loop space $\Omega SO(3)$, and we exhibit second homology generators for both connected components of $SF(S^1 \times S^2)$.
\end{abstract}

\section{Introduction}\label{sec:intro}


A \emph{Seifert-fibered 3-manifold} $M$ is a compact $S^1$-bundle over a 2-dimensional orbifold. This can be viewed as a compact 3-manifold that can be decomposed into a disjoint union of circles. 

\begin{remark}
For this paper, our fibers will be unoriented. 
\end{remark}

Let $M$ be a Seifert-fibered 3-manifold with Seifert fibration map $p: M \to \mathcal{O}$, where $\mathcal{O}$ is the base 2-orbifold. A \emph{Seifert fibering associated to $p$}, denoted $F_p$ is a decomposition into disjoint circles such that each circle is a fiber of $p$.

\begin{definition}
A \emph{fiber-preserving diffeomorphism} of a Seifert-fibered 3-manifold $M$ is a diffeomorphism that takes fibers of $F_p$ to fibers of $F_p$. The fiber-preserving diffeomorphisms form a subgroup $\aut(F_p) \subset Diff(M)$.
\end{definition}

\begin{definition}
We denote by $SF(M, p) = Diff(M)/\aut(F_p)$ the \emph{space of Seifert fiberings of $M$ equivalent to $F_p$}. Here $Diff(M)$ has the compact open topology, and the space has the natural structure of a coset space.
\end{definition}

\begin{remark}
This space can be interpreted as the space of circle foliations $F_p$ induced by $p: M \to \mathcal{O}$, i.e. each circle is a fiber of $p$. (Note that $p$ may have exceptional fibers.) Let $SF(M)$ be the union of all $SF(M, p)$ for all possible equivalence classes of Seifert fibrations $p$.
\end{remark}

Spaces of Seifert fiberings are known to have contractible components for many 3-manifolds, such as Haken Seifert-fibered 3-manifolds \cite{hkmr} and closed orientable Seifert-fibered 3-manifolds with hyperbolic base orbifolds \cite{mccullough}. In fact, for most compact aspherical orientable Seifert-fibered 3-manifolds $M$, it is known that $SF(M)$ has contractible components \cite{mccullough}. In this paper, we examine $SF(S^1 \times S^2, T)$, where $T: S^1 \times S^2 \to S^2$ is \emph{trivial Seifert fibration} that forgets the $S^1$ coordinate. We prove the following:

\begin{theorem}\label{thm:main}
The space $SF(S^1 \times S^2, T)$ is homotopy equivalent to a subspace $SF^{rigid}$ homeomorphic to $\Omega SO(3)$, the space of identity-based loops in $SO(3)$. This space has two connected components, each of which are homotopy equivalent to $\Omega(S^3)$. 
\end{theorem}


This theorem demonstrates that $S^1 \times S^2$ is a rare Seifert-fibered 3-manifold $M$ such that $SF(M)$ has noncontractible components. In addition, $S^1 \times S^2$ is the first known 3-manifold whose space of Seifert fiberings is not homotopy equivalent to a finite-dimensional CW-complex. (In \cite{dglmwy}, the space of Hopf fibrations of the 3-sphere is also shown to be homeomorphic to the disjoint union of two copies of $\mathbb{R}P^2$.)

\medskip

The proof of Theorem \ref{thm:main} will be based on Hatcher's determination of the homotopy type of $Diff(S^1 \times S^2)$ in \cite{hatcherpaper}, as well as the \emph{Cerf-Palais method} detailed in \cite{hkmr}. These tools will be discussed in Section \ref{sec:background}. 

\medskip

The space $\Omega SO(3)$ has two connected components, each homotopy equivalent to $\Omega S^3$. In addition, it is known (see \cite{hatcherpaper}) that 
\begin{equation}
	H_i(\Omega SO(3)) \cong \begin{cases}\mathbb{Z} & i \text{ even} \\ 0 & i \text{ odd}\end{cases}
\end{equation} 
Utilizing the reduced product suspension of James \cite{james}, we explicitly demonstrate a second homology generator for both connected components of $SF(S^1 \times S^2, T)$. The generators heavily involve Seifert fiberings associated to \emph{Gluck twists} $G_p$, which are rigid diffeomorphisms of $S^1 \times S^2$ which make one full rotation (around the $S^1$ factor) of the $S^2$ factor about an oriented axis associated to $p \in S^2$. The following is proven:

\begin{theorem}
Let $SF^{rigid}_T$, $SF^{rigid}_N$ be the connected components of $SF^{rigid}$ containing and not containing the trivial fibering, respectively. Let $s \in S^2$ be the south pole. Then $H_2(SF^{rigid}_T)$ is generated by $\{G_p * G_s \mid p \in S^2\}$ (here $*$ denotes concatenation), and $H_2(SF^{rigid}_N)$ is generated by $\{G_p \mid p \in S^2\}$.
\end{theorem}

\subsection{Outline of the paper}

Section \ref{sec:background} covers background on Smale-type theorems and spaces of Seifert fiberings, and sketches the method through which Theorem \ref{thm:main} is proven. Section \ref{sec:trivialfibrations} introduces the relevant space of Seifert fiberings, i.e. the space of ``rigid" Seifert fiberings of $S^1 \times S^2$. Section \ref{sec:defretract} completes the proof of Theorem \ref{thm:main}. Section \ref{sec:trivialcomponenth2} uses the James reduced product to visualize a second homology generator of the component of $SF(S^2 \times S^2, T)$ containing the ``trivial" fiberings. Section \ref{sec:nontrivialcomponenth2} introduces Gluck fiberings of $S^1 \times S^2$ and shows that the second homology of the other component of $SF(S^1 \times S^2, T)$ is generated by a sphere of Gluck fiberings. 

\subsection{Acknowledgements}

The authors would like to acknowledge Herman Gluck, Dennis DeTurck, Jim Stasheff, and Ziqi Fang for helpful discussions that spurred the writing of this paper. 

\section{The Cerf-Palais method}\label{sec:background}


We first review the existing literature on the study of spaces of fibrations of 3-manifolds with circle fiber. This theory is intimately connected to \emph{Smale-type theorems}. In general, a Smale-type theorem states that for a 3-manifold $M$, a subgroup of $Diff(M)$ (most often the isometry group) is homotopy equivalent to $Diff(M)$. For instance, the original Smale conjecture, proven by Hatcher in \cite{hatchersmaleconjecture}, stated that $Isom(S^3) \cong O(4) \hookrightarrow Diff(S^3)$ is a homotopy equivalence. Numerous papers have been written proving analogous theorems for many manifolds, culminating in the \emph{generalized Smale conjecture}, proven in \cite{bamlerkleiner}:

\begin{theorem}[\cite{bamlerkleiner} \cite{bamlerkleinerconnectsum}]
If $g$ is a Riemannian metric of constant sectional curvature $\pm 1$ on $X$, the inclusion $Isom(M, g) \hookrightarrow Diff(M, g)$ is a homotopy equivalence.
\end{theorem}

A similar theorem has been shown for nilmanifolds \cite{bamlerkleinernilmanifold}, as well as Seifert-fibered spaces with hyperbolic base orbifold \cite{mccullough}. The analogous theorem does not hold for all 3-manifolds, as shown by Hatcher:

\begin{theorem}[\cite{hatcherpaper}]
The diffeomorphism group $Diff(S^1 \times S^2)$ is homotopy equivalent to a subgroup homeomorphic to $O(2) \times O(3) \times \Omega SO(3)$, which is not isomorphic to $Isom(S^1 \times S^2)$.
\end{theorem}

%

Smale-type theorems are the main tool used to determine the homotopy types of spaces of Seifert fiberings. The general philosophy for using Smale-type theorems on $M$ to study the topology of $SF(M)$, which we will call the \emph{Cerf-Palais method} (stemming from \cite{hkmr}), is as follows. Let $M$ be a Seifert-fibered 3-manifold.
\begin{enumerate}
	\item Given a Seifert fibering $F_p$ induced by a fibration $p$, let $\aut(F_p) \subset Diff(M)$ be the subspace of diffeomorphisms that maps fibers of $F_p$ to fibers of $F_p$.
	\item Using Theorem 3.7.3 in \cite{hkmr}, construct the topological fiber bundle in the Frechet category:
	\begin{equation}
		\aut(F_p) \to Diff(M) \to SF(M, p)
	\end{equation}
	\item Prove or cite a Smale-type theorem to show that $Diff(M)$ is homotopy equivalent to a Frechet submanifold $Rigid(M)$, called the group of \emph{rigid diffeomorphisms of $M$}. 
	\item Let $SF^{rigid}(M, p)$ be the orbit of $F_p$ under $Rigid(M)$, and let $\aut^{rigid}(F_p)$ be the rigid diffeomorphisms that map fibers of $F_p$ to fibers of $F_p$. Construct a fiber bundle 
	\begin{equation}
		\aut^{rigid}(F_p) \to Rigid(M) \to SF^{rigid}(M, p)
	\end{equation}
	This step often requires some proof, depending on the involved spaces.
	\item Construct the fiber bundle diagram 
	\begin{equation}
		\begin{tikzcd}
			\aut(F_p) \arrow[r] & Diff(M) \arrow[r] & SF(M, p) \\
			\aut^{rigid}(F_p) \arrow[r] \arrow[u] & Rigid(M) \arrow[r] \arrow[u] & SF^{rigid}(M, p) \arrow[u]       
		\end{tikzcd}
	\end{equation}
	where the upper arrows are all inclusions.
	\item Construct the following sequence stemming from the long exact sequence of fiber bundles: 
	\begin{equation}\label{diag:longexact}
		\begin{tikzcd}
			\dots \arrow[r] & \pi_n(Diff(M)) \arrow[r]                            & \pi_n(SF(M, p)) \arrow[r] & \pi_{n-1}(\aut(F_p)) \arrow[r]                     & \dots \\
			\dots \arrow[r] & \pi_n(Rigid(M)) \arrow[r] \arrow[u] & \pi_n(SF^{rigid}(M, p)) \arrow[r] \arrow[u]       & \pi_{n-1}(\aut^{rigid}(F_p)) \arrow[r] \arrow[u] & \dots
		\end{tikzcd}
	\end{equation}
	\item Prove that the left upper arrow in Diagram \ref{diag:longexact} is a homotopy equivalence. 
	\item By the above step and the Smale-type theorem, all left and right arrows in the above diagram are isomorphisms. Use the Five Lemma to conclude that $SF^{rigid}(M, p) \to SF(M, p)$ is a weak homotopy equivalence.
	\item Conclude, using Theorem 3.7.3 in \cite{hkmr}, that $SF(M, p)$ is a separable topological Frechet manifold. Thus $SF(M, p)$ has the homotopy type of a CW-complex. 
	\item Show that $SF^{rigid}(M, p)$ is a CW complex.
	\item Conclude, using Whitehead's theorem, that $SF^{rigid}(M, p) \to SF(M, p)$ is a homotopy equivalence.
\end{enumerate}

The full Cerf-Palais method is usually not necessary. In fact, in many cases (e.g. Theorem 3.9.1 in \cite{hkmr}), $\aut(F_p) \to Diff(M)$ is a homotopy equivalence, which immediately implies that $SF(M, p)$ is contractible. This is the case for the results of \cite{hkmr} and \cite{mccullough}, which culminate in: 

\begin{theorem}[\cite{mccullough}]
Let $M$ be a compact orientable aspherical Seifert-fibered space other than a non-Haken infranilmanifold. Then each component of $SF(M)$ is contractible. 
\end{theorem}

In the case of non-aspherical manifolds, spaces of circle fibrations can have interesting topology. Upcoming work \cite{dglmwy} explores the space of Seifert fiberings of the three-sphere. The full Cerf-Palais method is used to prove the following theorem:

\begin{theorem}[\cite{dglmwy}]
The space of fiberings of $S^3$ by simple closed curves has the homotopy type of the disjoint union of a pair of two spheres if the fibers are oriented, and the disjoint union of a pair of two real projective planes if the fibers are unoriented.
\end{theorem}

\begin{remark}
This paper (as well as most cases in the literature) deals with spaces of non-oriented Seifert fiberings. If we were to consider oriented Seifert fiberings, the results would be slightly changed. The changes involve factors of $\mathbb{Z}/2\mathbb{Z}$. 
\end{remark}

The proof of Theorem \ref{thm:main} utilizes the full Cerf-Palais method using Hatcher's Smale-type result for $M = S^1 \times S^2$ and $p = T$, the trivial Seifert fibration. Notably, in this case, $Rigid(M)$ is not the isometry group of $M$. 


\section{Rigid Seifert fiberings of $S^1 \times S^2$}\label{sec:trivialfibrations}


Hatcher proved in \cite{hatcherpaper} that the diffeomorphism group $Diff(S^1 \times S^2)$ is homotopy equivalent to the subgroup $O(2) \times O(3) \times \Omega SO(3)$. The diffeomorphisms in this subgroup can be described as follows: given $(\alpha, \beta, \gamma_t) \in O(2) \times O(3) \times \Omega SO(3)$ and $(x, y) \in S^1 \times S^2$, we have the following:
\begin{equation}
(\alpha, \beta, \gamma_t) \cdot (x, y) = (\alpha(x), \beta(\gamma_x(y)))
\end{equation}
In other words, the $O(2)$ component rotates in the $S^1$-direction, the loop in $SO(3)$ dictates where the $S^2$-coordinate lands, and then the entire circles' worth of copies of $S^2$ are shifted by the $O(3)$ component. The order of the latter two actions is of note, as they do not commute.

\begin{definition}
Let $T$ denote the \emph{trivial Seifert fibering of $S^1 \times S^2$}, whose fibers are given by $\{(t, *), t \in S^1\}$ for any fixed $* \in S^2$. This is induced by the trivial fibration that forgets the $S^1$ coordinate, which is denoted $F_T$.
\end{definition}

In the context of the Cerf-Palais method, $O(2) \times O(3) \times \Omega SO(3)$ will be the rigid diffeomorphisms of $S^1 \times S^2$. With this in mind, we define:

\begin{definition}
The \emph{space of rigid Seifert fiberings of $S^1 \times S^2$}, denoted $SF^{rigid} \subset SF(S^1 \times S^2, T)$, is the orbit of $F_T$ under the action of the diffeomorphism subgroup $O(2) \times O(3) \times \Omega SO(3)$. 
\end{definition}

We first define a topology on $SF^{rigid}$ via a bijection $\Phi: \Omega SO(3) \to SF^{rigid}$.

\begin{definition}
Given an identity-based loop $\lambda_t \in \Omega SO(3)$, define a Seifert fibering $\Phi(\lambda_t)$ of $S^1 \times S^2$ by simple closed curves as follows. Start at a point $t_0 \in S^1$. Then the fiber of $\Phi(\lambda_t)$ crossing $(t_0, y) \in S^1 \times S^2$ is given by $(t, \lambda_t(y))$ as $t$ traverses $S^1$. 
\end{definition}

\begin{prop}\label{prop:bijection}
$\Phi: \Omega SO(3) \to SF^{rigid}$ is a bijection.
\end{prop}

\begin{proof}
If we have two distinct loops $\lambda \neq \gamma$ both in $\Omega SO(3)$, there exists some $(t, *) \in S^1 \times S^2$ such that $\lambda_t(*) \neq \gamma_t(*)$. Then $\Phi(\lambda)$ contains a fiber with the points $(t_0, *)$ and $(t, \lambda_t(*))$ while $\Phi(\gamma)$ contains a fiber with the points $(t_0, *)$ and $(t, \gamma_t(*))$. Thus $\Phi(\lambda)\neq \Phi(\gamma)$, and so $\Phi$ is injective. Now suppose we have a fibering $F \in SF^{rigid}$. By definition, this means that $SF^{rigid}$ is equal to the action of $(\alpha, \beta, \lambda) \in O(2) \times O(3) \times \Omega SO(3)$ on $F_T$, so each fiber is parameterized by $(\alpha(t), \beta(\lambda_t(*)))$ as $t$ varies throughout $S^1$, for each $* \in S^2$. Let $\zeta = \beta \circ \lambda_{\alpha^{-1}(t_0)}$, and let $\gamma_t = \beta \circ \lambda_{\alpha^{-1}(t)} \circ \zeta^{-1}$. Then the resulting fibering $\Phi(\gamma)$ will be parameterized by $(t, \beta \circ \lambda_{\alpha^{-1}(t)} \circ \zeta^{-1}(*))$. Notice that $\gamma_{t_0}$ is the identity, so $\gamma_t \in \Omega SO(3)$. In addition, up to reparameterization by $\alpha$ and $\zeta^{-1}$, this is the same fibering as the one parameterized by $(\alpha(t), \beta(\lambda_t(*)))$, i.e. $\Phi(\gamma) = F$. So $\Phi$ is surjective, and we have shown that $\Phi$ is a bijection.
\end{proof}

Endow $SF^{rigid}$ with the topology such that a set $U \subset SF^{rigid}$ is open if and only if $\Phi^{-1}(U)$ is open in the compact open topology on $\Omega SO(3)$. This makes $\Phi$ a homeomorphism. 

\begin{lemma}
\begin{equation}
	\aut(F_T) \cap (O(2) \times O(3) \times \Omega SO(3)) = O(2) \times O(3) \times \{id\}
\end{equation}
\end{lemma}

\begin{proof}
Since $O(3)$ acts by rotating along the $S^2$ coordinate of $S^1 \times S^2$, it permutes the fibers of $F_T$. Similarly, $O(2)$ rotates or reflects the fibers, which still preserves $F_T$. Therefore, $O(2) \times O(3) \times \{id\} \subset \aut(F_T) \cap (O(2) \times O(3) \times \Omega SO(3))$. Conversely, suppose $(\alpha, \beta, \lambda) \in \aut(F_T) \cap (O(2) \times O(3) \times \Omega SO(3))$. For contradiction, assume that $\lambda$ is not the trivial loop at the identity. By definition, there exists some $t \in S^1$ and $* \in S^2$ such that $\lambda_t(*) \neq *$. Then $(\alpha, \beta, \lambda)$ acting on the $F_T$-fiber $S^1 \times \{*\}$ can be parameterized by $(\alpha(t), \beta(\lambda_t(*))) \neq (\alpha(t), \beta(*))$. Both fibers share the common point $(\alpha(0), \beta(*))$, so they cannot be the same fiber. Therefore, $(\alpha, \beta, \lambda) \notin \aut(F_T)$. So, $\aut(F_T) \cap (O(2) \times O(3) \times \Omega SO(3)) \subset O(2) \times O(3) \times \{id\}$ and we have equality as desired.
\end{proof}

\begin{remark}
When viewing $O(2) \times O(3) \times \Omega SO(3)$ as a subgroup of $Diff(S^1 \times S^2)$, we use function composition as the group multiplication. This is not the same as the multiplication one would expect on $O(2) \times O(3) \times \Omega SO(3)$. It is not trivial to see that $O(2) \times O(3) \times \Omega SO(3)$ is a subgroup of $Diff(S^1 \times S^2)$ under composition. One way to see this is as follows: notice that $O(3) \times \Omega SO(3)$ is homeomorphic and isomorphic to the space of unbased loops in $O(3)$; call this Lie group $\Omega' O(3)$. Then let $(\alpha, \beta_t) \in O(2) \times \Omega'O(3)$ correspond to 
\begin{equation}   
	(x, y) \mapsto (\alpha(x), \beta_x(y))
\end{equation}
Then the group multiplication according to function composition is as follows: 
\begin{equation}
	(x, y) \xrightarrow{(\alpha, \beta_t)}(\alpha(x), \beta_x(y)) \xrightarrow{(\alpha', \beta_t')} (\alpha' \circ \alpha(x), \beta'_{\alpha(x)} \circ \beta_x(y))
\end{equation}
which means that 
\begin{equation}
	(\alpha', \beta'_t) \cdot (\alpha, \beta_t) = (\alpha' \cdot \alpha, \beta_{\alpha(t)}' \cdot \beta_t)
\end{equation}
where multiplication in $\Omega'O(3)$ is pointwise. To show associativity:
\begin{align}
	[(\alpha'', \beta_t'') \cdot (\alpha', \beta_t')] \cdot (\alpha, \beta_t) &= (\alpha'' \cdot \alpha', \beta_{\alpha'(t)}'' \cdot \beta'_t) \cdot (\alpha, \beta_t) \\ &= (\alpha'' \cdot \alpha' \cdot \alpha, \beta''_{\alpha' \cdot \alpha(t)} \cdot \beta_{\alpha(t)}' \circ \beta_t) \\ &= (\alpha'', \beta_t'') \cdot (\alpha' \cdot \alpha, \beta_{\alpha(t)}' \circ \beta_t) \\ &= (\alpha'', \beta_t'') \cdot (\alpha', \beta_t') \cdot (\alpha, \beta_t)
\end{align}
Finally, an inverse is given by 
\begin{equation}
	(\alpha, \beta_t)^{-1} = (\alpha^{-1}, \beta_{\alpha^{-1}(t)}^{-1})
\end{equation}
since 
\begin{equation}
	(\alpha, \beta_t) \cdot (\alpha^{-1}, \beta_{\alpha^{-1}(t)}^{-1}) = (\alpha \cdot \alpha^{-1}, \beta_{\alpha^{-1}(t)} \cdot \beta_{\alpha^{-1}(t)}^{-1}) = (\text{id}_{S^1}, \text{id}_{S^2})
\end{equation}
and 
\begin{equation}
	(\alpha^{-1}, \beta_{\alpha^{-1}(t)}^{-1}) \cdot (\alpha, \beta_t) = (\alpha^{-1} \cdot \alpha, \beta_{\alpha^{-1} \cdot \alpha(t)}^{-1} \cdot \beta_t = (\text{id}_{S^1}, \beta_t^{-1} \cdot \beta_t) = \text{id}_{S^1}, \text{id}_{S^2})
\end{equation}
By picking a basepoint $* \in S^1$, we see that the function 
\begin{equation}
	(x, y) \mapsto (\alpha(x), \beta_x(y))
\end{equation}
is equal to 
\begin{equation}
	(x, y) \mapsto (\alpha(x), \beta_* \circ (\beta_*^{-1} \circ \beta_x(y)))
\end{equation}
which corresponds to $(\alpha, \beta_*, \beta_*^{-1} \circ \beta_t) \in O(2) \times O(3) \times \Omega SO(3)$. Thus, the set of functions corresponding to $O(2) \times \Omega' O(3)$ is equal to the set of functions corresponding to $O(2) \times O(3) \times \Omega SO(3)$ inside $Diff(S^1 \times S^2)$. Then endow $O(2) \times O(3) \times \Omega SO(3)$ with the same subgroup structure as $O(2) \times \Omega'O(3)$. This group structure is very different from pointwise multiplication!
\end{remark}

\begin{lemma}\label{lma:selfhomeos}
Given $\alpha \in O(2), \beta \in O(3), \gamma \in \Omega SO(3)$, let $\zeta_{\alpha, \beta, \gamma} = \beta^{-1} \circ \gamma_{\alpha(t_0)} \circ \beta$. Then the map $c_{\alpha, \beta}: \Omega SO(3) \to \Omega SO(3)$ defined by $c_{\alpha, \beta}: \gamma_t \mapsto \beta^{-1} \circ \gamma_{\alpha(t)} \circ \beta \circ \zeta_{\alpha, \beta, \gamma}^{-1}$ is a homeomorphism.
\end{lemma}

\begin{proof}
The map $c_{\alpha, \beta}$ has an inverse given by $c_{\alpha, \beta}^{-1}: \gamma_t \mapsto \beta \circ \gamma_{\alpha^{-1}(t)} \circ \beta^{-1}$, so it is a bijection. Now suppose $U \subset \Omega SO(3)$ is an element of the subbase of the compact open topology, i.e. $U$ consists of maps that take some compact $K \subset S^1$ into an open $V \subset SO(3)$. Then $c_{\alpha, \beta}(U)$ consists of maps that take $\alpha(K)$ into $\beta V \beta^{-1}\zeta_{\alpha, \beta}^{-1}$; since $\alpha \in O(2)$ and $\beta \in O(3)$, $\alpha(K)$ is compact and $\beta V \beta^{-1}\zeta_{\alpha, \beta}^{-1}$ is open. Hence $c_{\alpha, \beta}$ is an open map. An exactly similar argument shows that $c_{\alpha, \beta}^{-1}$ is also open. Thus $c_{\alpha, \beta}$ is a homeomorphism.
\end{proof}

\begin{lemma}
The sequence 
\begin{equation}
	O(2) \times O(3) \to O(2) \times O(3) \times \Omega SO(3) \to SF^{rigid}
\end{equation}
is a fiber bundle.
\end{lemma}

\begin{proof}
We explicitly construct a local trivialization. Let 
\begin{equation}
	\mathfrak{p}: O(2) \times O(3) \times \Omega SO(3) \to SF^{rigid}
\end{equation}
be the latter map in the sequence, defined by the group action. Note that $\mathfrak{p}$ is an open map by Proposition \ref{prop:bijection} and the resulting topology on $SF^{rigid}$. Let $\mathcal{V} = \mathfrak{p}(O(2) \times SO(3) \times V')$, which is open. The local trivialization is a homeomorphism
\begin{equation}
	\phi: \mathfrak{p}^{-1}(\mathcal{V}) \to \mathcal{V} \times (O(2) \times O(3))
\end{equation}
Now define the local trivialization as follows: given $\alpha \in O(2), \beta \in SO(3), \gamma \in V'$, let $f = (\alpha, \beta, \gamma)$. Then 
\begin{equation}
	\phi(f) = (f(T), \alpha, \beta)
\end{equation}
We now show that this is a homeomorphism. Let $\gamma_t \in \Omega SO(3) = \Phi^{-1}(F)$. Then $F$ has fibers given by $(t, \gamma_t(*))$ for each $* \in S^2$. So let $\phi^{-1}(F, \alpha, \beta) = (\alpha, \beta, \beta^{-1} \circ \gamma_{\alpha(t)} \circ \beta \circ \zeta_{\alpha, \beta, \gamma}^{-1})$, where $\zeta_{\alpha, \beta, \gamma}$ is defined as in Lemma \ref{lma:selfhomeos}. The action of $\phi^{-1}(F, \alpha, \beta)$ on $F_T$ has fibers given by $(\alpha(t), \gamma_{\alpha(t)}(\beta(\zeta_{\alpha, \beta, \gamma}^{-1}(*))))$, which are the same as those given by $(t, \gamma_t(*))$. This shows that $\phi \circ \phi^{-1}$ is the identity. To show that $\phi^{-1} \circ \phi$ is the identity, note that given $(\alpha, \beta, \gamma_t)$ acting on $F_T$, the resulting fibering has fibers $(\alpha(t), \beta(\gamma_t(*))$, which are the same as those given by $(t, \beta \circ \gamma_{\alpha^{-1}(t)} \circ (\beta \circ \gamma_{\alpha^{-1}(t_0)})^{-1}(*))$. Thus the corresponding identity-based loop of this fibering is given by $\beta \circ \gamma_{\alpha^{-1}(t)} \circ (\beta \circ \gamma_{\alpha^{-1}(t_0)})^{-1}$. Note that when shifted by $\alpha$, this becomes $\beta \circ \gamma_t \circ \gamma_{\alpha^{-1}(t_0)}^{-1} \circ \beta^{-1}$. Therefore, $\phi^{-1} \circ \phi(\alpha, \beta, \gamma_t) = (\alpha, \beta, \beta^{-1} \circ (\beta \circ \gamma_t \circ \gamma_{\alpha^{-1}(t_0)}^{-1} \circ \beta^{-1}) \circ \beta \circ \zeta_{\alpha, \beta, \beta \circ \gamma_{\alpha^{-1}(t)} \circ (\beta \circ \gamma_{\alpha^{-1}(t_0)})^{-1}}^{-1}) = (\alpha, \beta, \gamma_t \circ  \gamma_{\alpha^{-1}(t_0)}^{-1} \circ (\beta^{-1} \circ (\beta \circ \gamma_{t_0} \circ \gamma_{\alpha^{-1}(t_0)}^{-1} \circ \beta^{-1}) \circ \beta)^{-1}) = (\alpha, \beta, \gamma_t \circ \gamma_{\alpha^{-1}(t_0)}^{-1} \circ \gamma_{\alpha^{-1}(t_0)})$, as desired. Thus $\phi$ is a bijection, and by Proposition \ref{prop:bijection} and Lemma \ref{lma:selfhomeos} it follows that $\phi^{-1}$ (and hence $\phi$) is a homeomorphism.
\end{proof}

\subsection{Proof of Theorem \ref{thm:main}}\label{sec:defretract}


By definition of $SF(S^1 \times S^2, T)$, $\text{Diff}(S^1 \times S^2)$ acts transitively on $SF(S^1 \times S^2, T)$. The next lemma follows from Theorem 3.7.3 in \cite{hkmr}.

\begin{lemma}
The sequence 
\begin{equation}
	\aut(F_T) \to \text{Diff}(S^1 \times S^2) \to SF(S^1 \times S^2, T)
\end{equation}
is a fiber bundle.
\end{lemma}

We thus obtain the fiber bundle diagram

\begin{equation}
\begin{tikzcd}
\aut(F_T) \arrow[r]                     & Diff(S^1 \times S^2) \arrow[r]                           & SF(S^1 \times S^2, T) \\
O(2) \times O(3) \arrow[r] \arrow[u] & O(2) \times O(3) \times \Omega SO(3) \arrow[r] \arrow[u] & SF^{rigid} \arrow[u]       
\end{tikzcd}
\end{equation}

By Hatcher's result, the middle arrow is a homotopy equivalence.

\begin{lemma}
The left arrow in the above diagram, i.e. $O(2) \times O(3) \to \aut(F_T)$, is a homotopy equivalence.
\end{lemma}

\begin{proof}
Since $S^1 \times S^2$ is a trivial $S^1$-bundle over $S^2$, we know that
\begin{equation}
	\aut(F_T) = C^\infty(S^2, Diff(S^1)) \times Diff(S^2)
\end{equation}
It is well-known \cite{Gluck} that $Diff(S^1)$ is smoothly homotopy equivalent to $O(2)$. As a result,
\begin{equation}
	C^\infty(S^2, O(2)) \hookrightarrow C^\infty(S^2, Diff(S^1))
\end{equation}  
is a homotopy equivalence. But every map $S^2 \to S^1$ is homotopically trivial, and can be deformed to a constant map in a canonical way, as follows: take any $f: S^2 \to S^1$ and lift to a map $\tilde{f}: S^2 \to \mathbb{R}$. Take the average value 
\begin{equation}
	A(\tilde{f}) = \frac{1}{4\pi}\int_{S^2}\tilde{f}(x)dx
\end{equation} 
and let $\tilde{f}_c$ denote the constant map taking all of $S^2$ to $A(\tilde{f})$. Then the homotopy 
\begin{equation}
	(1 - t)\tilde{f} + t\tilde{f}_c
\end{equation}
deforms $\tilde{f}$ to a constant map to $\mathbb{R}$. This is independent of the lift $S^1 \to \mathbb{R}$, and so project to $S^1$ to produce canonical deformations of smooth maps $S^2 \to S^1$ to constant ones. This is a smooth homotopy equivalence from $C^\infty(S^2, O(2))$ to $O(2)$. In addition, by \cite{smale}, $Diff(S^2)$ is homotopy equivalent to $O(3)$. In summary, we have shown that 
\begin{equation}
	\aut(F_T) = C^\infty(S^2, Diff(S^1)) \times Diff(S^2) \simeq O(2) \times O(3)
\end{equation}
as desired.
\end{proof}

\begin{remark}
The authors acknowledge Herman Gluck for commuting the above proof to us.
\end{remark}

This finally leads us to

\begin{theorem}
The right arrow in the above diagram, i.e. $SF^{rigid} \to SF(S^1 \times S^2, T)$, is a homotopy equivalence.
\end{theorem}

\begin{proof}
Take the fibration exact sequence in each fiber bundle. We end up with the following diagram: 
\begin{equation}
	\begin{tikzcd}
	\dots \arrow[r] & \pi_n(Diff(S^1 \times S^2)) \arrow[r]                            & \pi_n(SF(S^1 \times S^2, T)) \arrow[r] & \pi_{n-1}(\aut(F_T)) \arrow[r]                     & \dots \\
	\dots \arrow[r] & \pi_n(O(2) \times O(3) \times \Omega SO(3)) \arrow[r] \arrow[u] & \pi_n(SF^{rigid}) \arrow[r] \arrow[u]       & \pi_{n-1}(SO(2) \times O(3)) \arrow[r] \arrow[u] & \dots
	\end{tikzcd}
\end{equation}
By Hatcher's result and the above lemma, the left and right upward arrows are isomorphisms for all $n$. Then by the Five Lemma, for all $n$, the middle up arrow is also an isomorphism. Thus the middle arrow of the fiber bundle diagram $SF^{rigid} \to SF(S^1 \times S^2, T)$, is a weak homotopy equivalence. We know that $SF^{rigid} \cong \Omega SO(3)$, which is a CW complex (see \cite{milnor}). By Theorem 3.7.3 in \cite{hkmr}, $SF(S^1 \times S^2, T)$ is a smooth separable Frechet manifold, which is metrizable, and in particular has the homotopy type of a CW-complex. By Whitehead's theorem, $SF^{rigid} \to SF(S^1 \times S^2, T)$ is a homotopy equivalence, as desired. 
\end{proof}

\section{Second homology generators for $SF(S^1 \times S^2)$}

\subsection{The trivial component of $SF^{rigid}$}\label{sec:trivialcomponenth2}


Now that we have established that $SF^{rigid}$ is homotopy equivalent to $SF(S^1 \times S^2, T)$, it is interesting to look at its topological information. For instance, we know that $SF^{rigid}$ is homeomorphic to $\Omega SO(3)$. Since $SO(3) \cong \mathbb{R}P^3$, which is double covered by $S^3$, it follows that $\Omega SO(3)$ has two connected components, each of which is homotopy equivalent to $\Omega S^3$. We begin our investigation of the topology of $SF^{rigid}$ by describing the topology of $\Omega S^3$. It is well-known (e.g. from \cite{hatcherpaper}) that 
\begin{equation}
	H_i(\Omega S^3) = \begin{cases} \mathbb{Z} & i \text{ is even} \\ 0 & i \text{ is odd} \end{cases}
\end{equation}
We now describe homology generators for $\Omega SO(3)$, following the reduced product construction (see \cite{james}, \cite{hatcherbook}). 

\begin{definition}
The \emph{reduced product} $J(S^2)$, where $S^2$ has its basepoint $*$ at the north pole, is given by 
\begin{equation}
	J(S^2) = S^2 \sqcup (S^2)^2 \sqcup (S^2)^3 \sqcup \dots / \sim
\end{equation}
where
\begin{equation}
	(x_1, \dots, x_{k-1}, *, x_{k+1}, \dots, x_n) \sim (x_1, \dots, x_{k-1}, x_{k+1}, \dots, x_n)
\end{equation}
\end{definition}

We have the following useful result from Section 4.J of \cite{hatcherbook}.

\begin{theorem}[\cite{james}]
The loop space $\Omega S^{n+1}$ has the homotopy type of the James reduced product $J(S^n)$. 
\end{theorem}

View $S^3$ as the reduced suspension $\Sigma S^2 = (S^2 \times I) / (S^2 \times \partial I \cup \{*\} \times I)$ where $*$ is the basepoint of $S^2$. This can be pictured as below.

\begin{figure}[h]
	\centering
	\includegraphics[scale=.7]{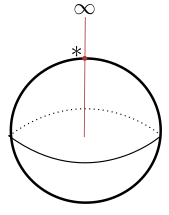}
\end{figure}

Now define a map $\lambda: S^2 \to \Omega S^3$ as follows. Given $x \in S^2$, the loop $\lambda_x$ is given by $\lambda_x(t) = (x, t) \in \Sigma S^2$. Note that this is a loop because $\lambda_x(0) = (x, 0) \sim (x, 1) = \lambda_x(1)$. In the picture, the loop $\lambda_x$ is depicted as follows, with $\lambda_x$ being the green ray, and the red ray being identified to a point as in the above picture.

\begin{figure}[h]
	\centering
	\includegraphics[scale=.7]{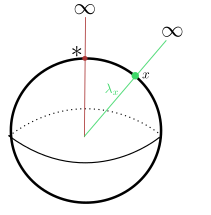}
\end{figure}

One can now define maps $\lambda^k: (S^2)^k \to \Omega S^3$ as follows. Given $(x_1, \dots, x_k) \in (S^2)^k$, we concatenate all the loops to get
\begin{equation}
	\lambda^k_{(x_1, \dots, x_k)}(t) = \begin{cases}\lambda_{x_1}(kt) & 0 \leq t \leq \frac{1}{k} \\ \lambda_{x_2}(kt - 1) & \frac{1}{k} \leq t \leq \frac{2}{k}\\ \dots \\ \lambda_{x_k}(kt - (k-1)) & \frac{k-1}{k} \leq t \leq 1\end{cases}
\end{equation}
Theorem 4J.1 in \cite{hatcherbook} shows that the map
\begin{equation}
\coprod_{k > 0} \lambda^k(x): J(S^2) \to \Omega S^3
\end{equation}
is a weak homotopy equivalence. The second homology group $H_2(\Omega SO(3))$ is thus generated by the image of $\lambda = \lambda^1: S^2 \to \Omega S^3$. 

\medskip

One can clearly see the image of $S^2$ in the reduced suspension picture of $S^3 = \Sigma S^2$. Given $x \in S^2$, the loop $\lambda_x$ can be viewed as the ray coming from the center and going through $x$, then proceeding in a straight line to infinity, which was identified with the center. Then $H_2(\Omega S^3)$ is given by $\{\lambda_x \mid x \in S^2\}$. 

\medskip

In order to find an explicit map to the component of $SF^{rigid}$ containing the trivial fibering, we'll want a more explicit picture of this generator of loops inside $S^3$, as opposed to the reduced suspension $\Sigma S^2$. We prove the following.

\begin{lemma}
Let $\lambda_x'$ be the loop in $S^3$ that goes from the center to the point at infinity, and then returns to the center vertically.

\begin{figure}[h]
	\centering
	\includegraphics[scale=.6]{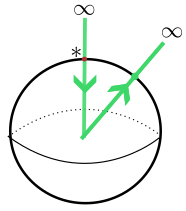}
	\caption{The loop $\lambda'_x$.}
\end{figure}
Then $\{\lambda_x' \mid x \in S^2\}$ generates $H_2(\Omega S^3)$. 
\end{lemma}

\begin{proof}
We already know that $\{\lambda_x \mid x \in S^2\}$ is the generator of $H_2(\Omega S^3)$ given by the James reduced product. A homotopy equivalence from the reduced product picture $S^3 = \Sigma S^2$ to $S^3$ is given by ``pulling" the point at infinity down to the center vertically. The image of $\lambda_x$ under this homotopy is a sphere of loops depicted as below.
\begin{figure}[h]\label{figure:generatorold}
	\centering
	\includegraphics[scale=.7]{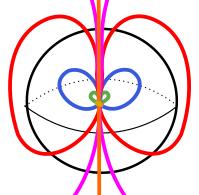}
	\caption{One homology generator of $H_2(\Omega S^3)$. This is homotopy equivalent to $\{\lambda_x' \mid x \in S^2\}$.}
\end{figure}
This can be seen to be homotopy equivalent to $\{\lambda_x' \mid x \in S^2\}$, by ``pulling" the loops toward infinity through the ray at $x \in S^2$. The identity loop (corresponding to the north pole) must be entirely extended to go toward infinity and then return, while the loop corresponding to the south pole does not change at all. In general, each loop corresponding to $x \in S^2$ is extended more and more toward infinity until it matches $\lambda_x'$, giving a homotopy equivalence. Thus $\{\lambda_x' \mid x \in S^2\}$ is a generator of $H_2(\Omega S^3)$, as desired.
\end{proof}

\begin{remark}
The following was pointed out to the authors by Rob Kusner. The homology generator depicted in the above figure can be seen in another way. The loop-suspension duality indicates that
\begin{equation}
	\pi_3(S^3) = \text{Map}(\Sigma S^2, S^3) = \text{Map}(S^2, \Omega S^3) = \pi_2(\Omega S^3)
\end{equation} 
Combining this with the the Hurewicz isomorphism $\pi_2(\Omega S^3) \to H_2(\Omega S^3)$ gives the same generator as in Figure \ref{figure:generatorold}. However, the James reduced product construction allows us to depict higher homology groups of $\Omega S^3$ if we wish, while the Hurewicz map does not. 
\end{remark}

Let $\Omega^T(SO(3))$ be the connected component of $\Omega SO(3)$ containing the trivial loop. Now that we have an explicit generator of $H_2(\Omega S^3)$, we use the double-cover of $S^3$ over $SO(3)$ to translate this into a generator of $H_2(\Omega^T(SO(3)))$. 

\medskip

%
%

In order to associate a loop in $SO(3)$ with a loop in $S^3$, we first project to loops in $\mathbb{R}P^3$. If the loop leaves the sphere, it gets reflected to the antipode. For example, the below loop gets projected to a loop in $\mathbb{R}P^3$, where $\mathbb{R}P^3$ is viewed as the solid 3-ball with antipodal boundary points identified. 

\begin{figure}[h]
	\centering
	\includegraphics[scale=.7]{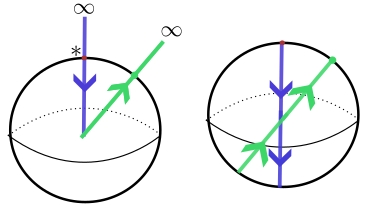}
	\caption{A loop in $S^3$ that goes to infinity via the green ray and comes back to the center via the blue ray, projected down to $\mathbb{R}P^3$.}
\end{figure}


Now we translate this to an element of $\Omega^TSO(3)$, where we view $SO(3)$ as the space of rotations of $S^2$. We take the following isomorphism from $\mathbb{R}P^3 \to SO(3)$. A point on $\mathbb{R}P^3$ is determined by $(x, t) \in S^2 \times [0, \pi]$, with $S^2 \times \{0\}$ being identified, and $(x, \pi) \sim (-x, \pi)$ for all $x \in S^2$. When mapping to $SO(3)$, let $x \in S^2$ be the oriented axis of rotation, and $t \in [0, \pi]$ be the amount of rotation. This homeomorphism is consistent with the equivalence relation, as $S^2 \times \{0\}$ (the center of the 3-ball) is always the identity, and antipodal points on the boundary sphere become rotations by $\pi$, where it does not matter whether the rotation is clockwise or counterclockwise.
%
%
%
%
%
%
%
\medskip

Let $SF^{rigid}_T$ be the image of $\Omega^T(SO(3))$ under $\Phi$, i.e. the component of $SF^{rigid}$ containing the trivial fibering. Utilizing the homeomorphism $\Phi|_{\Omega^TSO(3)}: \Omega^TSO(3) \to SF^{rigid}_T$, we can depict an explicit generator of the trivial component of $H_2(SF^{rigid}_T)$. We will picture $S^1 \times S^2$ as a thickened sphere $S^2 \times [0, 1]$, where the two boundary spheres should be identified. The trivial Seifert fibering $F_T$, under this model, looks like: 

\begin{figure}[h]
	\centering
	\includegraphics[scale=.5]{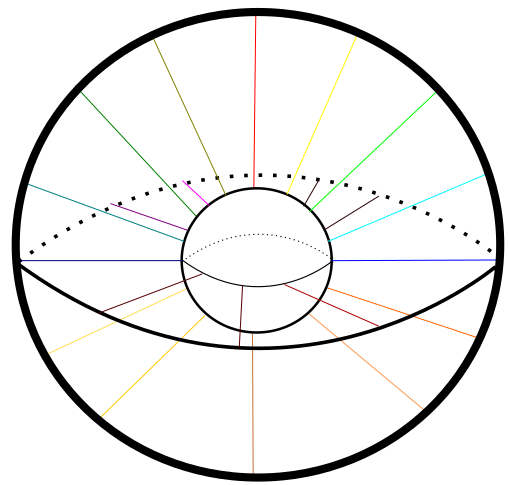}
	\caption{The trivial fibering, which corresponds to the north pole of the generator of $H_2(SF^{rigid})$.}
\end{figure}

The fibers go out radially from the central sphere. The generator of $H_2(SF^{rigid}_T)$ can be described as follows: 

\begin{corollary}
$H^2(SF^{rigid}_T)$ is generated by a sphere of fiberings that first perform a full rotation around the oriented axis corresponding to $x \in S^2$, and then performs a single rotation around the oriented axis corresponding to the south pole.
\end{corollary} 

The following is a example of one such point in the generator corresponding to the point $p$.

\begin{figure}[h]
	\centering
	\includegraphics[scale=.5]{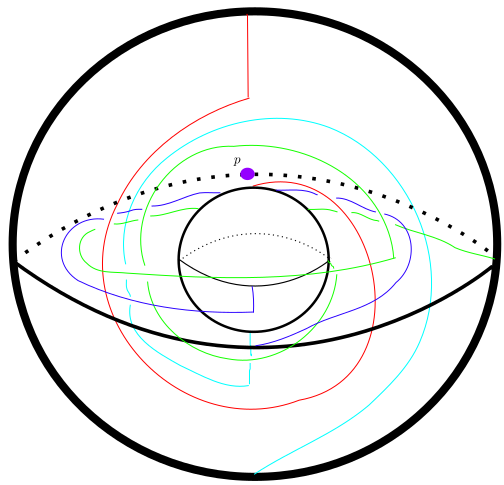}
\end{figure}

\subsection{Gluck twists and the nontrivial component of $SF^{rigid}$}\label{sec:nontrivialcomponenth2}


We are also able to depict a second homology generator of the nontrivial component of $SF^{rigid}$, denoted $SF^{rigid}_N$. We first describe the generator as follows. 

\begin{definition}
Given $p \in S^2$, let $\Lambda_t^p$ be the loop on $S^2$ that completes one full rotation around the oriented axis determined by $p$.
\end{definition}

\begin{remark}
The terminology comes from the \emph{Gluck twist} associated to $x$, which is a self-diffeomorphism of $S^1 \times S^2$ that sends $(x, y) \mapsto (x, \Lambda_x^p)$. Let this map be denoted $G_p$. We get an associated fibering $\Phi(G_p)$. Call this the \emph{Gluck fibering at $p$}.
\end{remark}

\begin{remark}
The Gluck fiberings correspond to the loop in $SO(3) \cong \mathbb{R}P^3$ given by the oriented axis at $p$ inside the ball model of $\mathbb{R}P^3$.
\end{remark}

The Gluck fibering at the north pole looks like the following: 

\begin{figure}[h]
	\centering
	\includegraphics[scale=.5]{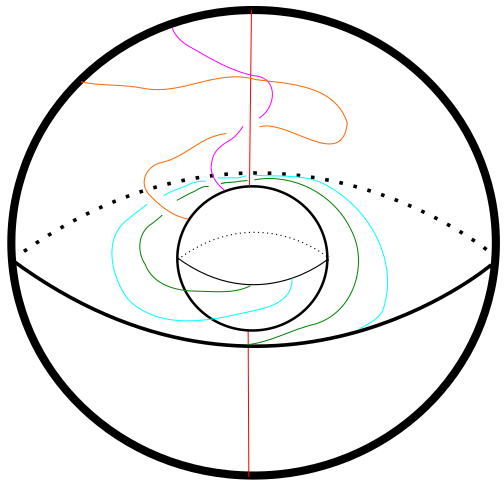}
\end{figure}

Note that the previous generator of $H_2(SF^{rigid}_T)$ can be described as first performing the Gluck fibering $\Lambda_t^p$ and then performing the Gluck fibering at the south pole, for all $p \in S^2$. In other words, the generator of $SF^{rigid}_T$ is given by $\{\Phi(G_p) * \Phi(G_s) \mid p \in S^2\}$, where $s \in S^2$ is the south pole and $*$ denotes concatenation. We will show that the homology generator of the nontrivial component of $SF^{rigid}_N$ can be described as $\{\Phi(G_p) \mid p \in S^2\}$, i.e. the sphere of Gluck fiberings. 

\begin{lemma}\label{lma:infinityloop}
Let $\Omega^\infty S^3$ be the space of paths from the basepoint $*$ to $\infty$, the point at infinity. Then the map $H: \Omega^\infty S^3 \to \Omega S^3$, given by concatenating with the vertical ray from $\infty$ to $*$, is a homotopy equivalence.  
\end{lemma}

\begin{proof}
This follows from \cite{hatcherbook}, Exercise 3 (page 291). 
\end{proof}

\begin{corollary}
The sphere of Gluck fiberings generates $H_2(\Omega SO(3))$. 
\end{corollary}

\begin{proof}
Let $\Omega^N(SO(3))$ be the connected component of $\Omega SO(3)$ that does not contain the trivial loop. Note that $\Omega^\infty S^3$ is homeomorphic to $\Omega^N(SO(3))$ via the quotient map. It follows from this and Lemma \ref{lma:infinityloop} that the map $h: \Omega^N(SO(3)) \to \Omega^T(SO(3))$ given by concatenating the downward vertical loop is a homotopy equivalence. The generator of $H_2(\Omega^T(SO(3)))$ is the image of the sphere of Gluck fiberings under $h$. Thus the sphere of Gluck fiberings generates  $H_2(\Omega^N(SO(3)))$.  
\end{proof}

\bibliography{bibliography.bib}

\begin{thebibliography}{10}

\bibitem{bamlerkleiner}
Richard Bamler and Bruce Kleiner.
\newblock Ricci flow and diffeomorphism groups of 3-manifolds, 08 2022.

\bibitem{bamlerkleinerconnectsum}
Richard~H. Bamler and Bruce Kleiner.
\newblock Ricci flow and contractibility of spaces of metrics, 2019.

\bibitem{bamlerkleinernilmanifold}
Richard~H. Bamler and Bruce Kleiner.
\newblock Diffeomorphism groups of prime 3-manifolds, 2023.

\bibitem{dglmwy}
Dennis DeTurck, Herman Gluck, Leandro Lichtenfelz, Mona Merling, Yi~Wang, and
  Jingye Yang.
\newblock Homotopy type of the moduli space of fibrations of the three-sphere
  by simple closed curves.
\newblock {\em in progress}.

\bibitem{Gluck}
Herman Gluck.
\newblock The embedding of two-spheres in the four-sphere.
\newblock {\em Trans. Amer. Math. Soc.}, 104:308--333, 1962.

\bibitem{hatcherpaper}
Allen Hatcher.
\newblock On the diffeomorphism group of {$S^1 \times S^2$}.
\newblock {\em Proceedings of the AMS 83}, 1981.

\bibitem{hatcherbook}
Allen Hatcher.
\newblock {\em Algebraic Topology}.
\newblock Allen Hatcher, 2001.

\bibitem{hatchersmaleconjecture}
Allen~E. Hatcher.
\newblock A proof of the smale conjecture, $\operatorname{Diff}(s^3) \simeq
  o(4)$.
\newblock {\em Annals of Mathematics}, 117(3):553--607, 1983.

\bibitem{hkmr}
Sungbok Hong, John Kalliongis, Darryl McCullough, and Hyam Rubinstein.
\newblock {\em Diffeomorphisms of Elliptic 3-Manifolds}, volume 2055.
\newblock 10 2011.

\bibitem{james}
I.M. James.
\newblock Reduced product spaces.
\newblock {\em The Annals of Mathematics, 2nd Ser., Vol. 62, No.1}, 1955.

\bibitem{mccullough}
Darryl McCullough and Teruhiko Soma.
\newblock The smale conjecture for seifert fibered spaces with hyperbolic base
  orbifold.
\newblock {\em J. Differential Geom.}, 93(2):327--353, 2 2013.

\bibitem{milnor}
John Milnor.
\newblock {\em Morse Theory}.
\newblock Princeton University Press, 1963.

\bibitem{smale}
Stephen Smale.
\newblock Diffeomorphisms of the 2-sphere.
\newblock {\em Proceedings of the American Mathematical Society Vol. 10, No.
  4}, 1959.

\end{thebibliography}
\bibliographystyle{plain}

\end{document}